\documentclass[11pt]{amsart}
\usepackage{amsmath,amsthm,amsfonts,amssymb,times}
\usepackage{latexsym,epsfig,subfigure,afterpage} % for graphics

\usepackage{titlesec}
\titleformat{\section}  %% large section titles
{\normalfont\Large\bfseries}
{\thesection}{1em}{}
\titleformat{\subsection}
{\normalfont\large\bfseries}
{\thesubsection}{1em}{}

\DeclareMathAlphabet{\curly}{U}{rsfs}{m}{n}

\theoremstyle{remark}

\theoremstyle{plain}

\newtheorem{lem}{Lemma}[section]
\newtheorem{thm}{Theorem}
\newtheorem{cor}{Corollary}
\numberwithin{equation}{section}

\newcommand{\ZZ}{{\mathbb Z}}
\newcommand{\QQ}{{\mathbb Q}}
\newcommand{\RR}{{\mathbb R}}
\newcommand{\CC}{{\mathbb C}}
\newcommand{\NN}{{\mathbb N}}

\newcommand{\bal}{\[\begin{aligned}}
\newcommand{\eal}{\end{aligned}\]}

\newcommand{\be}{\begin{equation}}
\newcommand{\ee}{\end{equation}}

\newcommand{\ssum}[1]{\sum_{\substack{#1}}}  %%% stacked sum

\newcommand{\benn}{\begin{equation*}}
\newcommand{\eenn}{\end{equation*}}

\newcommand{\lam}{\ensuremath{\lambda}}

\renewcommand{\a}{\ensuremath{\alpha}}
\renewcommand{\b}{\ensuremath{\beta}}
\newcommand{\g}{\gamma}
\newcommand{\del}{\ensuremath{\delta}}
\newcommand{\eps}{\ensuremath{\varepsilon}}

\newcommand\hh{\hat{h}}
\renewcommand{\SS}{\mathcal{S}}  % Selberg class
\newcommand{\HH}{\mathcal{H}}  % class of test functions 

\renewcommand{\(}{\left(}
\renewcommand{\)}{\right)}

\newcommand{\pfrac}[2]{\left(\frac{#1}{#2}\right)}
\newcommand{\order}{\asymp} % order of magnitude
\renewcommand{\le}{\leqslant}
\renewcommand{\leq}{\leqslant}
\renewcommand{\ge}{\geqslant}
\renewcommand{\geq}{\geqslant}

\newcommand{\sumstarrho}{{\sum_\rho}^*}

\begin{document}

%\title{A new repulsion phenomenon between zeros of $L$-functions}
\title{Unnormalized differences between zeros of $L$-functions}
%\title{Linear combinations of zeros of $L$-functions}
\author{Kevin Ford, Alexandru Zaharescu}
\date{\today}
\address{Department of Mathematics, 1409 West Green Street, University
of Illinois at Urbana-Champaign, Urbana, IL 61801, USA}
\email{ford@math.uiuc.edu}
\email{zaharesc@math.uiuc.edu}

\begin{abstract}
We study a subtle inequity in the distribution of unnormalized differences between imaginary parts of zeros
of the Riemann zeta function, which was observed by a number of authors.  
We establish a precise measure which explains the phenomenon,
that the location of each Riemann zero is encoded in the distribution
of large Riemann zeros.  We also extend these results to zeros of more general $L$-functions.
In particular, we show how the rank of an elliptic curve over $\mathbb{Q}$ is encoded in the 
sequences of zeros of {\it other} $L-$functions, not only the one associated to the curve.

\end{abstract}

\thanks{The first author was supported in part by National Science Foundation
grant DMS-1201442}

\maketitle

%%%%%%%%%%%%%%%%%%%%%%%%%

\section{Introduction}

%%%%%%%%%%%%%%%%%%%%%%%%%

The study of local spacing distribution of sequences
was initiated by physicists (see Wigner \cite{Wi} and Dyson \cite{Dy}) in order to understand the spectra of high energies. These notions have received a great deal of attention in many areas of mathematical physics, analysis, probability theory and number theory. After the pioneering 
work of Montgomery \cite{Mo} on the pair correlation of the imaginary parts of
zeros of the Riemann zeta function, followed by Hejhal's investigation \cite{He} of the triple correlation, higher level correlations for general $L-$functions have been studied by
Rudnick and Sarnak \cite{RS}, and Katz and Sarnak \cite{KS1}, \cite{KS2}. 
In particular, in the case of the Riemann zeta function, the above mentioned works establish
the existence for all $m$ of the limiting $m-$level correlations, for certain large classes of
test functions, and confirm that the limiting correlations are the same as those from the GUE model. One important feature of the GUE model is that the density function vanishes at the origin. 
This means strong repulsion between consecutive zeros of the Riemann zeta function
(see also supporting numerical computations by Odlyzko \cite{Od}).

More recently, the distribution of {\it unnormalized} differences between zeros have been considered
by a number of researchers.  Motivated in part by attempts to understand the limitations of the predictions
of random matrix theory, Bogomolny and Keating \cite{BoKe} were the first to make a percise conjecture for the 
discrepancy of such differences; see also the survery articles of Berry-Keating \cite{BeKe} and Snaith \cite{Sn}.
In particular, as pointed out in \cite{BeKe}, there seems to be a striking resurgence phenomenon: in the pair
correlation of high Riemann zeros, the low Riemann zeros appear as resonances; see also \cite[Fig. 3]{Sn} showing
the gaps between the first 100,000 Riemann zeros.

In 2011, P\'erez-Marco \cite{Ma} performed
extensive numerical studies of the distribution of differences of zeros of 
the Riemann zeta function and of other $L$-functions.
He also highlighted an interesting discrepancy phenomenon for zeros of $\zeta(s)$, namely
that the differences tend to avoid imaginary parts of small zeros, e.g. 
there is a deficiency of differences near $\gamma_1=14.1347\ldots$, the imaginary part of the smallest nontrivial
zero of $\zeta$. One may interpret this as saying that there is some repulsion
between the imaginary parts of zeros
of $\zeta(s)$, and the translates of these imaginary parts by $\gamma_1$.
The reader is referred to \cite{Ma} for extensive data supporting this phenomenon,
for $\zeta(s)$ as well as for more general $L-$functions.

In the present paper, our goal is to prove, unconditionally, 
a precise measure of the discrepancy of gaps in the distribution of unnormalized 
differences between zeros.
To proceed, let $\hh(\xi)=\int h(x) e^{-2\pi i x \xi}\, dx$ denote the Fourier transform of $h$. 
We will consider $h\in \HH$, the class of functions $h$ satisfying
\begin{enumerate}
 \item $\int h = 0$,
%\item the point 0 is not in the closure of the support of $h$,
\item $\int |x h(x)|\, dx < \infty$
\item $|\hh'(\xi)|\ll |\xi|^{-5}$ for $|\xi|\ge 1$.
\end{enumerate}

Let $\HH_0$ be the set of functions $h\in \HH$ such that in addition,
0 is not in the closure of the support of $h$.

\begin{thm}\label{main}
Suppose $h\in \HH_0$. For $T\ge 2$,
\[
 \sum_{0<\g,\g' \le T} h(\g-\g') = \frac{T}{2\pi} \int_{-\infty}^\infty h(t) \big(
K(1+it)+K(1-it) \big)\, dt + O\pfrac{T}{(\log T)^{1/3}},
\]
where $K(s)$ is defined as
\[
 K(s) = \sum_{n=1}^\infty \frac{\Lambda^2(n)}{n^s} \qquad (\Re s > 1),
\]
and by analytic continuation to $\{ s : \Re s \ge 1, s\ne 1 \}$.
Here $\g,\g'$ are imaginary parts of  nontrivial zeros  of $\zeta(s)$,
each zero occurring in the sum the number of times of its multiplicity.
\end{thm}

We mention here that the conclusion of Theorem \ref{main}, with suitable test functions $h$, has been obtained by
Conrey and Snaith \cite{CS} under the assumption of the $L$-functions ratios conjecture
of Conrey, Farmer and Zirnbauer \cite{CFZ}.  More recently, the conclusion of Theorem \ref{main}
was deduced, for test functions $h$ with compactly supported Fourier transforms and under the
assumption of the Riemann Hypothesis, by Rodgers \cite[Theorem 1.4]{Ro}.

{\bf Remarks.}  (i) The constant implied by the $O-$symbol depends on $h$.  
By standard counts of zeros (see \eqref{NT} and \eqref{NT1} below), if $h$ has compact support
then there are $\order T\log^2 T$ nonzero summands $h(\g-\g')$.  Thus, Theorem \ref{main}
implies that there is a discrepancy in the distribution of $\g-\g'$ of relative order 
$1/\log^2 T$.  The discrepancy has a ``density function'' $K(1+it)+K(1-it)$,
which is graphed in Figure 1.

(ii) $K(s)$ is close to the ``nice'' (from an analytic point of view) function
\[
 \pfrac{\zeta'}{\zeta}'(s) = \sum_{n=1}^\infty \frac{\Lambda(n) \log n}{n^s} \qquad (\Re s > 1).
\]
Let $b_m=\mu(ker(m)) \phi(ker(m))$, where $ker(m)=\prod_{p|m} p$.  Using the identity
$1=\sum_{m|k} (k/m)b_m$, we obtain the meromorphic continuation
\[
 K(s) = \sum_{m=1}^\infty b_m  \pfrac{\zeta'}{\zeta}'(ms) \qquad (\Re s > 0),
\]
provided that $ms\ne 1$ and
$\zeta(ms)\ne 0$ for all $m\in \NN$, the sum on $m$ converges absolutely, since 
$|b_m| \le m$ and $(\zeta'/\zeta)'(\sigma+it) = O(2^{-\sigma})$ for $\sigma\ge 2$.

Invoking the explicit formula for $(\zeta'/\zeta)(s)$ \cite[\S 12, (8)]{Da}
one gets, for any fixed $M\ge 4$ and $\sigma=\Re s > \frac12$,
\be\label{explicit}
K(s) = \sum_{m=1}^M \frac{b_m}{m^2} \( \frac{1}{(s-1/m)^2} - \sumstarrho \frac{1}{(s-\rho/m)^2}\) + 
O\( M 2^{-\sigma M} \).
\ee
where $\sumstarrho$ denotes a sum over \emph{all} zeros of $\zeta$, including the trivial zeros.

The terms in \eqref{explicit} corresponding to $m=1$ are the most significant.  Here $b_1=1$,
and we see that if $s=1+it$ and $t \approx \Im \rho$ for some $\rho$, the term in
\eqref{explicit} corresponding to $\rho$ will be 
$$-\frac{b_m}{m^2} \frac{1}{(s-\rho/m)^2} \approx -\frac{1}{(1-1/2)^2}=-4.$$
This accounts for the noticeable dips of about $-8$ in magnitude on the graph of $K(1+it)+K(1-it)$ (Figure 1)
corresponding to low-lying zeros  of $\zeta(s)$ (imaginary parts 
$\gamma_1=14.134\ldots, \gamma_2=21.022\ldots, \ldots$).
Likewise, for small $m > 1$, if $t \approx \Im \rho/m$ for some $\rho$
there will be a large positive or negative (depending on the sign of $b_m$)
term in \eqref{explicit} corresponding to $\rho$.  In particular, if $t\approx \Im \rho/2$,
the term in \eqref{explicit} corresponding to $\rho$ is about $-\frac{b_2}{2^2} \frac{1}{(1-1/4)^2}=\frac{4}{9}$, 
and if $t\approx \Im \rho/3$, the term corresponding to $\rho$ is about $\frac{8}{25}$.
  We have highlighted a few values of $t$
in Figure 1 where terms corresponding to $m=2$ and $m=3$ produce noticeable peaks 
in the graph.

\medskip
\medskip

\begin{figure}
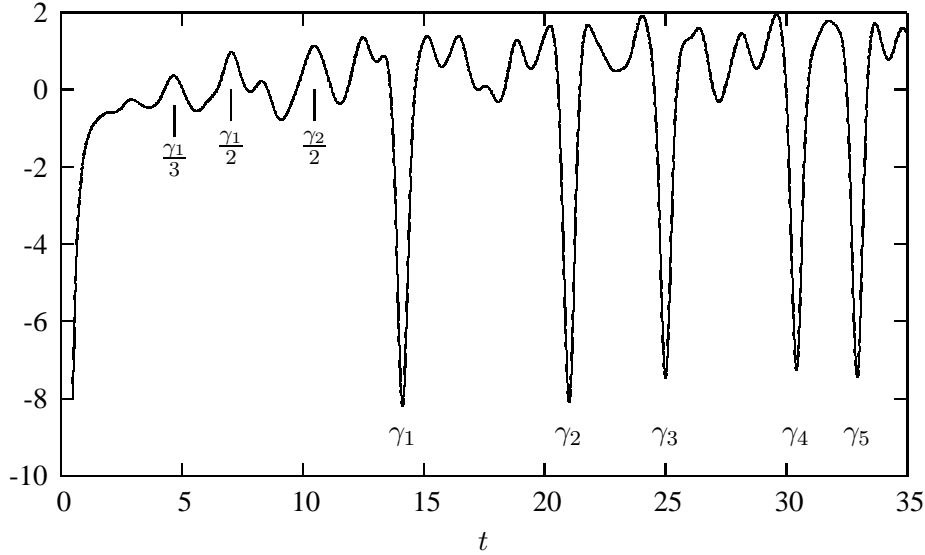

\begin{center}
\input gammaplot1-35.tex
\end{center}
\caption{Plot of $K(1+it)+K(1-it)$ together with small zeros of $\zeta(s)$}
\end{figure}

%
%   General L-functions
%

It is not difficult to generalize Theorem \ref{main} to the study of differences
of imaginary parts of zeros of a large class of $L$-functions.
Generally speaking, an \emph{$L$-function} is a Dirichlet series $F(s)=
\sum_{n=1}^\infty a_F(n) n^{-s}$ whose coefficients possess both 
\emph{multiplicative structure} (that is, $F$ has a Euler product) and \emph{additive structure}
(encoded in a functional equation equation for $F$).
One framework in which to operate is the Selberg Class $\SS$, the set of
$F$ satisfying the following axioms:
\par
 (i) there exists an integer $m \geq 0$ such that $(s-1)^mF(s)$
is an entire function of finite order; 
\par (ii)  $F$ satisfies
a functional equation 
\be\label{funceq}
\Phi(s) = w\overline{\Phi}(1-s), \qquad
\Phi(s) = Q^s \prod_{j=1}^r \Gamma(\lambda_j s + \mu_j) F(s) 
\ee
with $Q>0$, $\lambda_j > 0$, $\mu_j\in\CC$, $\Re(\mu_j) \geq 0$ and $|w|=1$.  (Here,
$\overline{f}(s) = \overline{f(\overline{s})}$); 
\par (iii) $F(s)$ has an Euler product ($a_F(n)$ is multiplicative), which we write as
$$
-\frac{F'}{F}(s)
=\sum_{n=1}^\infty \Lambda_F(n) n^{-s},
$$
where $\Lambda_F(n)$ is supported on powers of primes;
\par (iv) $\Lambda_F(n) \ll n^{\theta_F}$ for some $\theta_F
< \frac12$; and 
\par  (v) for every $\eps>0$, $a_F(n)\ll_\eps n^{\eps}$.

The functional equation is not uniquely determined in light of the
duplication formula for $\Gamma$-function, however the real sum
$d_F = 2 \sum_{j=1}^r \lam_j$
is well-defined and is known as the degree of $F$. 

For an introduction to results and conjectures concerning the
Selberg class, the reader may consult the survey papers of Kaczorowski and Perelli
\cite{KP}, Kaczorowski \cite{Ka} and Perelli \cite{Pe1,Pe2}.  
In particular, $\SS$ includes the Riemann zeta
function, Dirichlet $L$-functions, as well as $L$-functions attached to
number fields, elliptic curves, and holomorphic cusp forms. 
Subject to the truth of the open Ramanujan-Petersson conjecture (condition (v) above), the Selberg class
also includes primitive Maass cusp forms, and more generally, 
all automorphic $L$-functions (in the sense of Godement-Jacquet).
Chapter 5 of the book of Iwaniec and Kowalski \cite{IK} is another excellent
reference for the basic theory of $L$-functions.

By convention, all sums, products and counts of zeros of a function $F\in \SS$
count zeros with their multiplicity.   A generic zero will be written as
$\rho=\beta+i\gamma$ (the variable with or without subscripts), with $\b,\g \in \RR$.
We denote $Z(F)$ the multi-set of nontrivial zeros (those with $0 < \b < 1$).
All sums over zeros will include only nontrivial zeros unless otherwise noted.

To generalize Theorem \ref{main}, we 
will require two additional hypotheses on $F$, a mild Mertens-type estimate for $|\Lambda_F(n)^2|$
and a zero-density estimate for $F$ near $\Re s = \frac12$.  It is expected that all zeros 
in $Z(F)$ have real part $\frac12$ (the analog of Riemann's
Hypothesis for $\zeta(s)$).  
\begin{enumerate}
\item[(vi)] We have 
\[
 \sum_{n\le x} \frac{|\Lambda_F(n)|^2}{n} \ll \log^2 x;
\]
 \item[(vii)]\label{ZF} There exist constants $A(F)>0, B(F)>0$ such that 
\[
N_F(\sigma,T) := \left| \left\{ \rho=\beta+i\gamma\in Z(F): \beta \ge \sigma,
 0<\gamma \le T \right\} \right| \ll T^{1-A(F)(\sigma-1/2)}(\log T)^{B(F)},
\]
uniformly for $\sigma\ge 1/2$ and $T\ge 2$. 
\end{enumerate}
Let $\SS^*$ denote the set of $F\in \SS$ satisfying (vi) and (vii) above.
Condition (vi) is very mild and holds in practice, in particular for all Dirichlet $L$-functions 
(trivially from the bound $|\Lambda(n)|\le \log n$) and also for all automorphic $L$-functions (see 
Theorem \ref{KFG} below).  Condition (vii), is known for 
the Riemann zeta function and Dirichlet $L$-functions (Selberg \cite{S1}, \cite{S2} with $B(F)=1$),
and  degree 2 $L$-functions attached to newforms on the full modular group (Luo \cite{Luo}; 
also with $B(F)=1$).  As we will see later in Section \ref{sec:Li}, the method of Luo
also can be used to show (vii) for newforms attached to congruence subgroups.

\begin{thm}\label{main2}
 Let $F,G\in \SS^*$ and suppose further that the Dirichlet series
\[
 K_{F,G}(s) = \sum_{n=1}^\infty \frac{\Lambda_F(n) \overline{\Lambda_G(n)}}{n^s}
\]
can be analytically continued to the region $\{ s: \Re s \ge 1, s\ne 1\}$.
For any $h\in \HH_0$ with compact support,
\be\label{sumh}
 \ssum{0<\g \le T \\ \rho \in Z(F)} \ssum{0<\g' \le T \\ \rho' \in Z(G)} 
h(\g-\g')
= \frac{T}{\pi} \Re \int_{-\infty}^\infty h(t)  K_{F,G}(1+it) \, dt +
O\pfrac{T}{(\log T)^{1/3}}.
\ee
Furthermore, if $K_{F,G}$ is analytic at $s=1$, then \eqref{sumh} holds also for all $h\in\HH$
with compact support.
\end{thm}

\textbf{Remarks.}  The hypothesis that $h$ have compact support may be omitted, provided that the function
$K_{F,G}$ satisfies $|K_{F,G}(\sigma+it)| \ll 1+|t|$ uniformly for $\sigma\ge 1$, (excluding $|t|\ll 1$ if $K_{F,G}$
has a pole at $s=1$).  See the proof of \eqref{finale} at the end of Section 4. 
This is the case, e.g. if $F$ and $G$ are both Dirchlet $L$-functions.
In general, in order to ensure that the integral \eqref{sumh} converges absolutely,
one needs a hypothesis on the decay of $h(t)$ as $|t|\to \infty$ and/or a hypothesis on the growth
of $K_{F,G}(\sigma+it)$ for $\sigma+it$ as $|t|\to\infty$.

In general, the function $K_{F,G}(s)$ should satisfy the hypotheses of Theorem \ref{main2}, although 
it is not known to be true for every $F,G\in S$. 
The function $K_{F,G}$ is closely related to the Rankin-Selberg convolution of the $L$-functions
$F$ and $G$.
In the special case where $\chi$ and $\psi$
are Dirichet characters (principal characters allowed), $F(s)=L(s,\chi)$, and $G(s)=L(s,\psi)$, we have
$\Lambda_F(n)\overline{\Lambda_G(n)}=\chi(n)\overline{\psi(n)}\Lambda^2(n)$ and hence
\[
 K_{F,G}(s) = \frac{d^2}{ds^2} \log L(s,\chi \overline{\psi}) + \widetilde{K}(s),
\]
where $\widetilde{K}(s)$ is a Dirichlet series, absolutely convergent and analytic for $\Re s > 1/2$.
By classical arguments \cite[\S 14, main theorem]{Da}, $L(s,\chi\overline{\psi})\ne 0$ for $\Re s \ge 1$
and thus $K_{F,G}(s)$ is analytic on $\{s : \Re s \ge 1, s\ne 1\}$.  An elaboration of this argument 
gives the following, which we will prove later in Section \ref{sec:KFG}.  The proof is based on
properties of Rankin-Selberg convolution $L$-functions.

\begin{thm}\label{KFG}
(a) Condition (vi) holds for any automorphic $L$-function $F$;
(b) Suppose $F$ and $G$ are automorphic $L$-functions, of arbitrary degree.  Then $K_{F,G}$ is
  analytic on $\{s : \Re s \ge 1, s\ne 1\}$.  Furthermore, $K_{F,G}$ is analytic at $s=1$
  unless $F=G$.
\end{thm}

It is natural to ask about more general linear combinations of zeros in \eqref{sumh}. 
Let $\a$ be an arbitrary positive real number.
Our method of proof of Theorem \ref{main2} indeed yields the more general result that
\[
 \ssum{0<\g \le T \\ \rho \in Z(F)} \ssum{0<\g' \le T \\ \rho' \in Z(G)} 
h(\g-\a\g')
= \frac{T}{\pi} \Re \int_\RR h(t)  K^{(\a)}_{F,G}(1+it) \, dt +
O\pfrac{T}{(\log T)^{1/3}},
\]
where
\[
 K^{(\a)}_{F,G}(s) = \sum_{n=1}^\infty \frac{\Lambda_F(n) \overline{\Lambda_G(n^\a)}}{n^s} \qquad (\Re s >1).
\]
Here $\Lambda_G(x)=0$ unless $x\in\NN$.  We observe that the sum on $n$ is empty unless $\a = \frac{\log a}{\log b}$
for some positive integers $a,b$, and that $K_{F,G}^{(\a)}$ is trivially analytic for $\Re s=1$ if
$\a$ is irrational.

\bigskip
%A function $F\in \SS$ is said to be \emph{primitive} if it cannot be
%written as a product of two or more elements of $\SS$.  Clearly it
%suffices to prove our results in the case of $F$ being primitive,
%which we henceforth assume [DO WE REALLY NEED THIS?].  

\medskip

We emphasize that the local minima of $K(1+it)+K(1-it)$ (see Figure 1) 
are not necessarily attained exactly at $t=\gamma_j$.
However, statistics on gaps between zeros are capable of identifying the exact location of every zero.
This comes from working with higher derivatives of $K(s)$, in which the influence of each zero 
is sharpened.
This is confirmed by our 
next theorem, which shows that the Riemann Hypothesis for the Riemann zeta function
is encoded in the (distribution of) zeros of every single Dirichlet $L-$function.

\begin{thm}\label{main3}
 Fix an integer $q > 1$ and a primitive Dirichlet character $\chi$ mod $q$. Fix a
 compactly supported function $g$ in $\mathcal{C}^\infty(\mathbb{R})$ such that
 $g \ge 0$ and $\int g = 1$. Fix a real number $\beta > \frac12$. Then, for any positive
 integer $k$, the function $f_{\chi,g,\beta,k}:(\frac12, \infty) \rightarrow \mathbb{R}$
 given by
 \[
 f_{\chi,g,\beta,k}(\alpha) := \frac{(-1)^{k+1}}{2^{k+1}(k+1)!} \lim_{m\rightarrow\infty}
  \lim_{T\rightarrow\infty}
  \frac{2\pi}{T}\ssum{0<\g,\g' \le T \\ L(\rho,\chi) = L(\rho',\chi) = 0} 
h_m^{(k)}(\g-\g') ,
 \]
is well defined, where
\[
h_m(x) := m \Big( g(m(x-\alpha)) - g(m(x-\beta)) \Big).
\]
Moreover, the following are equivalent:

(a) The sequence of functions $(f_{\chi,g,\beta,k})_{k\ge1}$ converges (pointwise) on
$(\frac12, \infty)$,

(b) The Riemann Hypothesis holds true for $\zeta(s)$.

More precisely, assuming  the Riemann Hypothesis for $\zeta(s)$, for any $\beta>\a>1/2$,
\[
 \lim_{k\rightarrow\infty} f_{\chi,g,\beta,k}(\alpha) = m_{\zeta}\left(\frac12 + i\alpha\right) -
 m_{\zeta}\left(\frac12 + i\beta\right) ,
 \]
 where $m_{\zeta}(\rho)$ denotes the multiplicity of the zero
 of the Riemann zeta function at $\rho$.
\end{thm}

Theorem \ref{main3} is an instance where a property involving the distribution of
zeros of an $L-$function is proved to be equivalent to the Riemann Hypothesis for 
another $L-$function.  One can use the methods employed to extend
Theorem \ref{main} to Theorem \ref{main2} 
 in order to
obtain a generalization of Theorem \ref{main3}.  We discuss only one such application here,
to illustrate the ideas further.

For two general $L$-functions $F,G\in \SS$ we define
\[
  f_{F,G,g,\beta,k}(\alpha) := \frac{(-1)^{k+1}  }{2^{k+1}(k+1)!} \lim_{m\rightarrow\infty}
  \lim_{T\rightarrow\infty}
  \frac{2\pi}{T}\ssum{0<\g,\g' \le T \\ F(\rho) = G(\rho') = 0} 
h_m^{(k)}(\g-\g'),
\]
where $g$ and $h_m$ are the same as in Theorem \ref{main3}.

Let $E$ be an elliptic curve over the rational field $\QQ$ with %nonzero 
discriminant $\Delta$
and conductor $N$.  For squarefree $d$, let $E_d$ be the quadratic twist of $E$; e.g., if $E$ is
given by the equation
\[
 E: \, y^2 = x^3 + ax + b,
\]
then $E_d$ is given by
\[
 E_d: \,  dy^2 = x^3 + ax + b.
\]
Let $L(s,E)$ be the (normalized) $L$-function attached to $E$.  By Hasse's theorem, we have
\be\label{LsE}
L(s,E) = \sum_{n=1}^\infty a(n) n^{-s} = \prod_p \(1-\a_p p^{-s}\)^{-1}\(1-\b_p p^{-s}\)^{-1} \quad (\Re s >1),
\ee 
where $|\a_p|=|\b_p|=1$ and $\a_p =\overline{\b_p}$ for $p\nmid \Delta$, and for $p|\Delta$,
$0\le \a_p\le 1$ and $\b_p=0$.
That $L(s,E)$ lies in the Selberg class (modularity) is a celebrated result of Wiles et al 
\cite{Wiles1}, \cite{TW}, \cite{BCDT}.
In fact, $L(s,E)$ coincides with the $L$-function of a weight-2 newform on the congruence subgroup 
$\Gamma_0(N)$.  The same holds for all quadratic twists $L(s,E_d)$.  The Dirichlet coefficients of 
$L(s,E_d)=\sum_{n=1}^\infty a_d(n) n^{-s}$ are %simply 
given by $a_d(n)=\chi_d(n) a(n)$, where
$\chi_d$ is the primitive quadratic character modulo $d$ (the Jacobi symbol modulo $d$). 
 Now let
\be\label{FGHW}
 F(s) = L(s,E_d), \qquad G(s)=L(s,\chi_d), \qquad H(s)=L(s,E), \qquad W(s)=\pfrac{H'}{H}'(s).
\ee

%The result of Luo \cite{Luo}, which proves (vii) for cusp forms on the full modular group,
%also holds for congruence subgroups; see Section \ref{sec:Li}.
We can then apply Theorem \ref{main2} and obtain that the distribution of 
differences $\g-\g'$, with $F(\rho)=G(\rho')=0$, is related to the behavior of $K_{F,G}(1+it)$.
The latter function is closely associated with the zeros of $H$.  Of particular interest is the
multiplicity $m_E(1/2)$ of the zero at the point $s=1/2$, which caries important arithmetic information.
 By Mordell's theorem the group
$E(\QQ)$ of rational points on $E$ is finitely generated,
\[
E(\QQ) \widetilde = \ZZ^r + E(\QQ)^{\text{tors}},
\]
where $E(\QQ)^{\text{tors}}$ is a finite abelian group.
According to the Birch and Swinnerton-Dyer Conjecture \cite{BS} (see also \cite{Wiles2}), 
the rank $r$ of $E$ equals the multiplicity
$m_E(1/2)$.

\begin{thm}\label{main5}
Assume that $d$ is squarefree,  and
that all zeros of $H(s)$ with imaginary part $\le \max(|\a|,|\b|)+1$ have real part equal
to $\frac12$. Then 
 \[
 \lim_{k\rightarrow\infty} f_{F,G,g,\beta,k}(\alpha) = m_{E}\left(\frac12 - i\alpha\right) -
 m_{E}\left(\frac12 - i\beta\right) ,
 \]
 where $m_{E}(\rho)$ denotes the multiplicity of the zero of $H(s)$ at $s=\rho$.
 \end{thm}

Taking $\a=0$ and $\b$ an ordinate where $H(1/2+i\b)\ne 0$, the right side above equals
the ``analytic rank'' $m_E(\frac12)$, which conjecturally equals
the rank of $E(\QQ)$.  The left side, on the other hand, is an expression involving only the distribution
of zeros of $L(s,E_d)$ and zeros of $L(s,\chi_d)$. Conceptually what this means is that
some arithmetic information on the given elliptic curve $E$ is encoded in the sequences
of zeros of {\it other} $L-$functions, not only the one associated to $E$.

%%%%%%%%%%%%%%%%%%%%%%%%%

\section{Classical sums over zeros}

%%%%%%%%%%%%%%%%%%%%%%%%%

Assume that $F\in \SS^*$.
We use frequently the estimate (cf. \cite{S3}, (1.6))
\be\label{NT}
N_F(T) = \left| \{ \rho=\beta+i\gamma \in Z(F) : 0<\beta<1,
0<\gamma\le T\} \right| = \frac{d_F}{2\pi} T \log T + c_1 T + O(\log T)
\ee
for some constant $c_1=c_1(F)$.
An easy consequence is
\be\label{NT1}
N_F(T+A)-N_F(T) = O(A\log T) \qquad (1\le A\le T).
\ee
Another easy corollary is, for $F\in \SS^*$ and $G\in \SS^*$,
\be\label{sumrecipgamma}
\ssum{0<\g,\g'\le T \\ \rho\in Z(F), \rho'\in Z(G)} \min \( K, \frac{1}{|\g-\g'|} \) \ll 
T(K+\log T)\log^2 T \qquad (K \ge 1).
\ee

Finally, we need a uniform version of Landau's theorem \cite[Lemma 1]{FZ}.
The following is given in \cite{FSZ}.

\begin{lem}\label{Landau} Let $F\in\mathcal S$, $x>1$, $T\geq2$, and 
let $n_x$ be a nearest integer to $x$.
Then, for any $\eps > 0$,
\benn\label{c1}
\sum_{0<\gamma\leq T}x^{\rho}=
-\frac{\Lambda_F(n_x)}{2\pi} \frac{e^{iT\log(x/n_x)}-1}{i\log(x/n_x)}
+O_{\eps}\(x^{1+\theta_F}\log (2x) + x^{1+\eps} \log T + \frac{\log T}{\log x} \).
\eenn
\end{lem}

\textbf{Remark.}  If $1<x<3/2$, then $n_x=1$, $\Lambda_F(n_x)=0$ and the error term in Lemma \ref{Landau}
is $O(1+\frac{\log T}{\log x})$.  For $x=1+o(1/T)$, this is worse than the trivial bound $O(T\log T)$ coming from
\eqref{NT}.  Thus, we may replace the term $\frac{\log T}{\log x}$ in the statement of Lemma \ref{Landau} by
$\min(T\log T, \frac{\log T}{\log x})$.

%%%%%%%%%%%%%%%%%%%%%%%%%%%%%%%%%%%%%%%%%%%%%%%%%%%%%

\section{The contribution of zeros off the critical line}

%%%%%%%%%%%%%%%%%%%%%%%%%%%%%%%%%%%%%%%%%%%%%%%%%%%%%

For $F\in \SS^*$, define
\be\label{DFxi}
\begin{split}
 D_F(\xi) = D_F(\xi;T) &= \ssum{\rho\in Z(F) \\ 0<\g\le T} 
e^{2\pi i\g \xi}\( 1 - e^{2\pi \xi(\beta-1/2)} \) \\
 &= \frac12 \sum_{0<\g\le T} e^{2\pi i\g \xi}\( 2 - e^{2\pi \xi(\beta-1/2)} - 
e^{2\pi \xi(1/2-\beta)} \).
\end{split}
\ee
The second equality follows from the fact that if $\beta+i\g$ is a nontrivial
zero of $F$ then so is $1-\beta+i\g$ (this follows from the functional equation (ii)).
Throughout the remainder of this section, $\rho_j=\beta_j+i\g_j\in Z(F)$ ($j=1,2$).

\begin{lem}\label{beta}
Suppose that $F\in \SS^*$.
Uniformly for $T\ge 3$ and $1\le U \le \frac{A(F)}{8\pi} \log T$,
\[
 \int_0^U |D_F(\xi;T)|^2\, d\xi \ll \frac{U^4 T}{\log T}(\log\log T)^4 + T^{1-A(F)/U} (\log T)^{B(F)+2}. 
\]
\end{lem}

\begin{proof}
We have
\be\label{sumg1g2}
\int_0^U |D_F(\xi;T)|^2\, d\xi =\frac14 \sum_{0 < \g_1, \g_2 \le T} \int_0^U
 e^{2\pi i\xi(\g_1-\g_2)} \prod_{k=1}^2 (2-e^{2\pi\xi(\b_k-1/2)}-e^{2\pi\xi(1/2-\b_k)}) \, d\xi.
\ee
Let $I(\rho_1,\rho_2,\xi)$ be the integrand on the right side of \eqref{sumg1g2}, and put
\[
 \Delta= \g_1-\g_2, \qquad \a = \max\big(|\b_1-1/2|, |\b_2-1/2|\big).
\]

We consider four conditions on the pair $(\rho_1,\rho_2)$.

Case 1: $|\Delta|\le 1$ and $\a \le \frac{1}{U}$.  Since $|2-e^x-e^{-x}|\ll x^2$ for $0\le x\le 2\pi$, we have
\[
 I(\rho_1,\rho_2,\xi)\ll \a^4 \xi^4 \le \a^4 U^4.
\]
 Let
\[
 N_F^*(\sigma,T) = \# \{ (\rho_1,\rho_2): 0<\g_1,\g_2\le T: \max(\b_1,\b_2) \ge \sigma,
|\g_1-\g_2|\le 1 \}.
\]
From \eqref{NT1}, we have
\[
 N_F^*(\sigma,T) \ll (\log T) N_F(\sigma,T).
\]
The sum of $\int_0^U I(\rho_1,\rho_2,\xi)\, d\xi$ over all such pairs
$(\rho_1,\rho_2)$ is, using \eqref{NT} and the zero-density estimate (vi),
\begin{align*}
 &\ll U^5 \int_0^{1/U} \a^4 (-dN_F^*(\tfrac12+\a,T)) \le 3U^5 \int_0^{1/U}
\a^3 N_F^*(\tfrac12+\a,T)\, d\a \\
&\ll U^5 T \int_0^{\infty} \a^3 \min\( T^{-\a A(F)} (\log T)^{B(F)+1}, \log^2 T\) \, d\a \ll
\frac{U^5 T}{\log^2 T}(\log\log T)^4.
\end{align*}

Case 2:  $|\Delta|\le 1$ and $\a > \frac{1}{U}$.   Then
$I(\rho_1,\rho_2,\xi)\ll e^{4\pi \a \xi}$ and thus 
\[
 \int_0^U I(\rho_1,\rho_2,\xi)\, d\xi \ll \a^{-1} e^{4\pi \a U} \le U e^{4\pi \a U}.
\]
The contribution to the right side of \eqref{sumg1g2} from all such pairs $(\rho_1,\rho_2)$ is
therefore
\begin{align*}
&\ll U \int_{1/U}^{1/2} e^{4\pi \a U} (-dN_F^*(\tfrac12+\a,T)) \\
&= U\left[ e^{4\pi} N_F^*(\tfrac12+\tfrac{1}{U},T) + 4\pi U \int_{1/U}^{1/2} 
e^{4\pi \a U} N_F^*(\tfrac12+\a,T)\, d\a \right] \\
&\ll U T (\log T)^{B(F)+1}  \left[ T^{-A(F)/U} + U\int_{1/U}^{\infty} e^{(4\pi U-(\log T)A(F))\a}\, d\a \right] \\
&\ll T^{1-A(F)/U} \(\log T\)^{B(F)+2}.
\end{align*}
We used repeatedly the upper bound $U\le \frac{A(F)}{8\pi} \log T$.

Case 3: $|\Delta|>1$ and $\a \le \frac{1}{U}$.  We have
\be\label{intI}
\int_0^U I(\rho_1,\rho_2,\xi)\, d\xi = \frac{1}{2\pi} 
\left[ 2g(0)-2g(\b_1-\tfrac12)-2g(\b_2-\tfrac12)+g(\b_1+\b_2-1)+g(\b_1-\b_2)\right],
\ee
where
\[
 g(z) = \frac{e^{2\pi i U \Delta+2\pi Uz}-1}{i\Delta+z} + \frac{e^{2\pi i U \Delta-2\pi Uz}-1}{i\Delta-z}.
\]
Since we apply this with $z\le \frac{1}{U}$, by Taylor's theorem we get
\begin{align*}
 g(z) &= e^{2\pi i U\Delta} \frac{2i\Delta +(4\pi^2 i U^2\Delta-4\pi U)z^2 + O(|\Delta| z^4 U^4)}
{-\Delta^2-z^2} + \frac{2i\Delta}{\Delta^2+z^2} \\ 
&= e^{2\pi i U\Delta} \frac{2i\Delta+(4\pi^2 i U^2\Delta-4\pi U-2i/\Delta)z^2+
O(|\Delta| z^4 U^4)}{-\Delta^2}+\frac{2i}{\Delta}\(1-\frac{z^2}{i\Delta^2}+O\pfrac{z^4}{\Delta^4}\).
\end{align*}
For any even quadratic polynomial $f:\CC\to\CC$, $2f(0)-2f(a)-2f(b)+f(a+b)+f(a-b)=0$.
Hence, by \eqref{intI},
\[
\int_0^U I(\rho_1,\rho_2,\xi)\, d\xi \ll \frac{\a^4 U^4}{|\Delta|}.
\]
For non-negative integer $j$, define
\[
 N_j(\sigma,T)=\# \{ (\rho_1,\rho_2) : 2^j < |\g_1-\g_2|\le 2^{j+1}, 0<\g_1,\g_2\le T, \beta_1\ge \sigma\}
\]
and observe that by \eqref{NT},
\be\label{Nj}
 N_j(\sigma,T) \ll (2^j\log T)N_F(\sigma,T).
\ee
Using condition (vii) in the definition of $\SS^*$, the contribution to the right side of \eqref{sumg1g2} 
from all pairs $(\rho_1,\rho_2)$ counted in case 3 is
\begin{align*}
 &\ll \sum_{1\le 2^j\le T} \frac{U^4}{2^j} \int_0^{1/U} \a^4 (-d N_j(\tfrac12+\a,T))\\
&\ll \sum_{1\le 2^j\le T} \frac{U^4}{2^j}\int_0^{1/U} \a^3 N_j(\tfrac12+\a,T)\, d\a\\
&\ll \sum_{1\le 2^j\le T} U^4 T  \int_0^{\infty} \a^3 \min\( T^{-\a A(F)} (\log T)^{B(F)+1},
\log^2 T\)\, d\a
  \ll \frac{U^4 T}{\log T}(\log\log T)^4.
\end{align*}

Case 4: $|\Delta|>1$, $\a > \frac{1}{U}$.
Using formula \eqref{intI}, the crude bound $|g(z)| \ll \frac{1}{|\Delta|} (
e^{2\pi U |z|}+1)$, and estimate \eqref{Nj},  the contribution to the right side of \eqref{sumg1g2} 
from all such pairs $(\rho_1,\rho_2)$ is
\begin{align*}
 &\ll \sum_{1\le 2^j\le T} 2^{-j} \int_{1/U}^{1/2} e^{4\pi U \a}
   (-dN_j(\tfrac12+\a,T)) \\
&= \sum_{1\le 2^j\le T} 2^{-j} \left[ e^{4\pi} N_j(\tfrac12+\tfrac{1}{U},T) + 4\pi U
  \int_{1/U}^{1/2} e^{4\pi U\a} N_j(\tfrac12+\a,T)\, d\a \right]\\
&\ll T (\log T)^{B(F)+1} \sum_{1\le 2^j\le T}  
\left[ T^{-A(F)/U} + \frac{U}{\log T} T^{-A(F)/U} \right] \ll T^{1-A(F)/U} (\log T)^{B(F)+2}.
\end{align*}

Combining the estimates in the four cases completes the proof of the lemma.
\end{proof}

%%%%%%%%%%%%%%%%%%%%%%%%%

\section{Proof of Theorems \ref{main} and \ref{main2}}

%%%%%%%%%%%%%%%%%%%%%%%%%

Throughout this section, $\rho_1=\b_1+i\g_1\in Z(F)$ and  $\rho_2=\b_2+i\g_2\in Z(G)$.

Writing
\[
 Q_F(\xi;T) = \ssum{\rho\in Z(F) \\ 0<\g\le T} e^{2\pi i \xi \g},
\]
we then have
\be\label{main-1}
S : = \sum_{0<\g_1,\g_2\le T} h(\g_1-\g_2) = \int_\RR \hh(\xi) Q_F(\xi;T) \overline{Q_G(\xi;T)}\, d\xi
=2 \Re \int_0^\infty\hh(\xi) Q_F(\xi;T) \overline{Q_G(\xi;T)}\, d\xi.
\ee
Decompose the integral as
\be\label{S1S2}
S=2\Re \int_{0}^U \hh(\xi)  Q_F(\xi;T) \overline{Q_G(\xi;T)}\, d\xi+ 
\int_{|\xi|>U} \hh(\xi)  Q_F(\xi;T) \overline{Q_G(\xi;T)}\, d\xi 
=:2\Re S_1+S_2,
\ee
where
\[
U = \frac{c\log T}{\log\log T}, \quad c = \frac{\min(1,A(F),A(G))}{100\pi}.
\]
We can quickly dispense with $S_2$ due to the rapid decay of $\hh'(\xi)$ from condition (3) in the definition
of $\HH$.   Since
\[
\int_W^V  Q_F(\xi;T) \overline{Q_G(\xi;T)}\,  d\xi \ll \sum_{0<\g_1,\g_2\le T} 
\min\( V-W, \frac{1}{|\g_1-\g_2|} \), 
\]
we obtain from \eqref{sumrecipgamma} and integration by parts the bound
\be\label{S2}\begin{split}
\Big| \int_U^\infty \hh(\xi)  Q_F(\xi;T) \overline{Q_G(\xi;T)}\, d\xi \Big| &\ll
\int_U^\infty |\hh'(\xi)| \sum_{0< \g_1,\g_2\le T} \min\(\xi,\frac{1}{|\g_1-\g_2|}\)\, d\xi \\
&\ll \frac{T\log^3 T}{U^4}+\frac{T\log^2 T}{U^3} \ll \frac{T (\log\log T)^4}{\log T}.
\end{split}\ee

For $S_1$, we decompose $Q_F(\xi;T)$ as follows:
\[
 Q_F(\xi;T) = \ssum{\rho\in Z(F) \\ 0<\g\le T} \left[ e^{2\pi \xi(\rho-1/2)} + e^{2\pi i \g \xi}\(
 1 - e^{2\pi \xi(\b-1/2)} \) \right] = M_F(\xi) + E_F(\xi) + D_F(\xi),
\]
where $D_F(\xi)$ is given by \eqref{DFxi}, and writing $x=e^{2\pi \xi}$,
\[
 M_F(\xi) = -\frac{\Lambda_F(n_x)}{2\pi\sqrt{n_x}} \frac{e^{iT\log(x/n_x)}-1}{i\log(x/n_x)}.
\]
By Lemma \ref{Landau} (cf. the remark following the lemma) and condition (iv) defining the Selberg class ($|\Lambda_F(n)|\ll
n^{\theta_F}$),
\bal
E_F(\xi) &= -\frac{\Lambda_F(n_x)}{2\pi} \frac{e^{iT\log(x/n_x)}-1}{i\log(x/n_x)} \( \frac{1}{\sqrt{x}}-\frac{1}{\sqrt{n_x}} \)
 + O\(x^{1/2+\theta_F+\eps}\log T+\min\(\frac{\log T}{\log x},T\log T\) \) \\
&\ll |\Lambda_F(n_x)| \frac{|x-n_x|}{n_x^{3/2} |\log(x/n_x)|} + e^{2 \pi \xi} \log T + \min \( \frac{\log T}{\xi},T\log T \) \\
&\ll e^{2 \pi \xi} \log T + \min \( \frac{\log T}{\xi},T\log T \) .
\eal
Think of $M_F(\xi)$ as the main term and $D_F(\xi)$ and $E_F(\xi)$ as error terms.
Making a similar decomposition $Q_G(\xi;T) = M_G(\xi) + D_G(\xi) + E_G(\xi)$,
we have
\begin{multline}
 \big| Q_F(\xi;T)\overline{Q_G(\xi;T)} - M_F(\xi)\overline{M_G(\xi)} \big|  = 
\big| M_F(\xi)(\overline{E_G(\xi)+D_G(\xi)})+\overline{M_G(\xi)}(E_F(\xi)+D_F(\xi))  \\ +
(E_F(\xi)+D_F(\xi))(\overline{E_G(\xi)}+\overline{D_G(\xi)}) \big|.
\end{multline}
Applying the Cauchy-Schwartz inequality, we obtain
\be\label{IEDM}
 \int_0^U |\hh(\xi)| \cdot \big| Q_F(\xi;T)\overline{Q_G(\xi;T)} - M_F(\xi)\overline{M_G(\xi)} \big|\, d\xi \ll
\mathcal{E}+\mathcal{D}+\mathcal{M}^{1/2} (\mathcal{E}^{1/2}+\mathcal{D}^{1/2}),
\ee
where
\bal
\mathcal{D}&=\int_0^U |\hh(\xi)| \( |D_F(\xi)|^2 + |D_G(\xi)|^2 \)\, d\xi, \\
\mathcal{E}&=\int_0^U |\hh(\xi)| \( |E_F(\xi)|^2 + |E_G(\xi)|^2 \)\, d\xi, \\
\mathcal{M}&=\int_0^U |\hh(\xi)| \( |M_F(\xi)|^2 + |M_G(\xi)|^2 \)\, d\xi.
\eal
Since $\hh(0)=0$, we have $|\hh(\xi)| \ll \xi$  and hence
\be\label{E}\begin{split}
 \mathcal{E} &\ll \int_0^{1/T}  \xi(T\log T)^2\, d\xi + \int_{1/T}^U \xi \left[
\pfrac{\log T}{\xi}^2+e^{4\pi\xi}\log^2 T \right]\, d\xi \\
&\ll \log^2 T + \log^3 T + U e^{4\pi U}\log^2 T \ll T^{1/2}.
\end{split}\ee
Observe that the hypothesized decay $|\hh'(\xi)|\ll |\xi|^{-5}$ for $|\xi|\ge 1$ implies that
\be\label{hhatdecay}
|\hh(\xi)| \ll |\xi|^{-4} \qquad (|\xi| \ge 1).
\ee
By \eqref{hhatdecay}, Lemma \ref{beta} and integration by parts,
\be\label{D}\begin{split}
 \mathcal{D} &\ll \int_0^U (1+\xi)^{-5} \int_{0}^\xi |D_F(v)|^2+|D_G(v)|^2\, dv\, d\xi \\
&\ll \frac{T(\log\log T)^4}{\log T}\int_0^U (1+\xi)^{-1}\, d\xi \\
&\ll \frac{T(\log\log T)^5}{\log T}.
\end{split}\ee
Next, using condition (vi) in the definition of $\SS^*$,
\be\label{M}\begin{split}
 \mathcal{M} &\ll 1 + \sum_{2\le n\le e^{2\pi U}+1} \frac{|\Lambda_F(n)|^2+|\Lambda_G(n)|^2}{n}
\int_{\frac{\log(n-1/2)}{2\pi}}^{\frac{\log(n+1/2)}{2\pi}}
|\hh(\xi)| \min\( T, \frac{1}{|2\pi \xi-\log n|}\)^2\,d\xi\\
&\ll 1 + \sum_{2\le n\le e^{2\pi U}+1} \frac{T(|\Lambda_F(n)|^2+|\Lambda_G(n)|^2)}{n(\log n)^4} \ll T.
\end{split}\ee
Thus, by \eqref{main-1}, \eqref{S1S2}, \eqref{S2},\eqref{IEDM}, \eqref{E}, \eqref{D} and \eqref{M},
\be\label{main-2}
\begin{split}
S &= 2\Re \int_0^U \hh(\xi) M_F(\xi)\overline{M_G(\xi)} d\xi + 
  O\pfrac{T}{(\log T)^{1/3}}\\
&= \frac{1}{2\pi^2} \Re \sum_{2\le n\le e^{2\pi U}+1} \frac{\Lambda_F(n)\overline{\Lambda_G(n)}}{n} 
\int_{\log(1-\frac{1}{2n})}^{\log(1+\frac{1}{2n})} \hh\pfrac{\log n + u}{2\pi} \left|
\frac{e^{iTu}-1}{iu} \right|^2\, du +O\pfrac{T}{(\log T)^{1/3}}.
\end{split}
\ee
Because $\hh'$ is bounded (a consequence of condition (3) in the definition of $\mathcal{H}$),
replacing $\hh(\frac{\log n + u}{2\pi})$ by $\hh(\frac{\log n}{2\pi})$ produces a 
total error $\mathcal{R}$ which satisfies
\bal
\mathcal{R} &\ll \sum_{2\le n\le e^{2\pi U}+1} \frac{|\Lambda_F(n)|^2+|\Lambda_G(n)|^2}{n} \int_{\log(1-\frac{1}{2n})}^{\log(1+\frac{1}{2n})}
 |u| \min \( T, \frac{1}{|u|} \)^2\, du \\
&\ll (\log T) \sum_{2\le n\le e^{2\pi U}+1} \frac{|\Lambda_F(n)|^2+|\Lambda_G(n)|^2}{n} \ll \log^3 T.
\eal
This implies that
\[
 S =  \frac{1}{2\pi^2} \Re \sum_{2\le n\le e^{2\pi U}+1} \frac{\Lambda_F(n)\overline{\Lambda_G(n)}}{n} 
\hh \pfrac{\log n}{2\pi} 
\int_{\log(1-\frac{1}{2n})}^{\log(1+\frac{1}{2n})} \left|
\frac{e^{iTu}-1}{iu} \right|^2\, du
+O\pfrac{T}{(\log T)^{1/3}}.
\]
Replacing the limits of integration with $(-\infty,\infty)$ produces an error of
\[
\sum_{n\le e^{2\pi U}+1} |\Lambda_F(n)|^2+|\Lambda_G(n)|^2 \ll T^{1/2},
\]
 while the ``complete'' integral equals $2\pi T$.  We obtain
\[
  S =  \frac{T}{\pi} \Re \sum_{2\le n\le e^{2\pi U}+1} \frac{\Lambda_F(n)\overline{\Lambda_G(n)}}{n} 
\hh \pfrac{\log n}{2\pi} +O\pfrac{T}{(\log T)^{1/3}}.
\]
Using (vi) once again, we extending the sum on $n$ to  all positive integers $n$.
The error $\mathcal{R}'$ induced satisfies 
\[
\mathcal{R}' \ll T \sum_{n>e^{2\pi U}} \frac{|\Lambda_F(n)|^2+|\Lambda_G(n)|^2}{n\log^4 n} \ll \frac{T}{U^2},
\]
hence
\[
  S =  \frac{T}{\pi} \Re \sum_{n=1}^\infty \frac{\Lambda_F(n)\overline{\Lambda_G(n)}}{n} 
 \hh \pfrac{\log n}{2\pi} +O\pfrac{T}{(\log T)^{1/3}}.
\]
 Finally, writing
$\hh$ as an integral and interchanging the sum and integral, we will show that
\be\label{finale}
 \sum_{n=1}^\infty \frac{\Lambda_F(n)\overline{\Lambda_G(n)}}{n} 
\hh \pfrac{\log n}{2\pi}  = \int_\RR h(t)  K_{F,G}(1+it) \, dt.
\ee
To show \eqref{finale}, we first observe that (vi) and the decay of $\hh(\xi)$ imply that
the left side of \eqref{finale} is absolutely convergent, and thus
\begin{align*}
\sum_{n=1}^\infty \frac{\Lambda_F(n)\overline{\Lambda_G(n)}}{n} 
\hh \pfrac{\log n}{2\pi} &=
 \lim_{\del\to0^+} \sum_{n=1}^\infty \frac{\Lambda_F(n)\overline{\Lambda_G(n)}}{n^{1+\del}} 
\hh \pfrac{\log n}{2\pi} \\
&= \lim_{\del\to0^+} \sum_{n=1}^\infty \frac{\Lambda_F(n)\overline{\Lambda_G(n)}}{n^{1+\del}}
\int_{-\infty}^\infty h(t) n^{-it}\, dt.
\end{align*}
For each \emph{fixed} $\del>0$, the sum-integral on the right side above is also absolutely convergent by 
(vi) and the assumption that $h\in L^1(\RR)$.  Hence, by Fubini's theorem, we may interchange the sum and integral,
 obtaining
\begin{align*}
 \sum_{n=1}^\infty \frac{\Lambda_F(n)\overline{\Lambda_G(n)}}{n} 
\hh \pfrac{\log n}{2\pi} &= \lim_{\del\to 0^+} \int_{-\infty}^\infty h(t) \sum_{n=1}^\infty
\frac{\Lambda_F(n)\overline{\Lambda_G(n)}}{n^{1+\del+it}} \, dt\\
&= \lim_{\del\to 0^+}\int_{-\infty}^\infty h(t)  K_{F,G}(1+\del+it) \, dt.
\end{align*}
Here we used the fact that the Dirichlet series for $K_{F,G}(s)$ converges to $K_{F,G}(s)$ for $\Re s > 1$.
Finally, we must take the limit back inside the integral.  We observe that by assumption,
$K_{F,G}(s)$ is analytic in $\{ s : \Re s \ge 1, s\ne 1 \}$.
 In the case of $F(s)=G(s)=\zeta(s)$ in Theorem 1,
the explicit formula \eqref{explicitzeta} plus standard density bounds for the zeros immediately give
\[
 K_{F,G}(1+\del+it) \ll \frac{1}{|\del+it|^2} + \log^3(2+|t|)
\]
uniformly for $\del\ge 0$.  Since $h\in \HH_0$, 
\[
 \int_{\-\infty}^\infty h(t) |K_{F,G}(1+\del+it)|\, dt
\]
converges uniformly for $\del>0$, and this proves \eqref{finale}.  For general $F$ and $G$, we do not necessarily
have a good enough zero-free region to deduce strong uniform bounds on $K_{F,G}(1+\del+it)$ as $\del\to 0^+$.
But here we assume that $h$ has compact support, so that the analyticity of $K_{F,G}$ still allows us to 
conclude \eqref{finale}.
  This completes the proof of Theorem \ref{main2}.

%%%%%%%%%%%%%%%%%%%%%%%%%%%%%%%%%%%%%%%%%%%%%%%%%%%%%%%%%%%%%%%%%%%%%%%%%%%%%%%%%%%%%555

\section{Analyticity of $K_{F,G}$.  Proof of Theorem \ref{KFG}}\label{sec:KFG}

%%%%%%%%%%%%%%%%%%%%%%%%%%%%%%%%%%%%%%%%%%%%%%%%%%%%%%%%%%%%%%%%%%%%%%%%%%%%%%%%%%%

Let $F$ and $G$ be automorphic $L$-functions on $GL(n)$ and $GL(m)$, respectively, where $n$
and $m$ are positive integers.  For basic analytic properties of automorphic $L$-functions, the reader may consult
Sections 5.11 and 5.12 of \cite{IK}.  In particular, the Euler products for $F$ and $G$ have the shape
\[
 F(s)=\prod_p \prod_{i=1}^n (1-\a_i(p)p^{-s})^{-1}, \qquad G(s)=\prod_p \prod_{j=1}^m (1-\b_j(p)p^{-s})^{-1}
\]
for complex constants $\a_i(p),\b_j(p)$.  Both Euler products are absolutely convergent for $\Re s > 1$,
and both functions have analytic continuation to $\CC \setminus \{1\}$ and satisfy
functional equations of type \eqref{funceq} (see e.g. \cite[\S 5.12]{IK} and the references therein). 
 Although the Ramanujan-Petersson
conjecture (which states that $|\a_i(p)|\le 1$, $|\b_j(p)|\le 1$ for all $p$, $i$ and $j$)
is not known for $L$-functions of degree exceeding 2, we do have the Luo-Rudnick-Sarnak
bounds \cite[(5.95)]{IK}
\be\label{LRS}
|\a_i(p)| \le p^{\frac12 - \frac1{n^2+1}}, \qquad |\b_j(p)| \le p^{\frac12 - \frac{1}{m^2+1}}.
\ee
The Rankin-Selberg convolution $L$-function of $F$ and $\overline{G}$, which we denote by $H(s)$,
has Euler product of the form \cite[(5.9),(5.10)]{IK}
\be\label{Hs}
H(s)=\prod_{p\nmid Q} \prod_{i=1}^n \prod_{j=1}^m (1-\a_i(p)\overline{\b_j}(p)p^{-s})^{-1} \prod_{p|Q}
\prod_{j=1}^{mn}(1-\gamma_j(p)p^{-s})^{-1},
\ee
for some positive integer $Q$ and numbers $\gamma_j$ satisfying $|\gamma_j|<p$. 
The function $H(s)$ also has absolutely convergent
Euler product for $\Re s > 1$ (and hence is nonzero in this open half-plane) \cite{JS}.
Also, Shahidi proved that $H(1+it)\ne 0$ for all real $t$ (\cite{Sh81}; with an announcement in \cite{Sh80}).
Moreover, $H(s)$ is analytic at $s=1$ unless $F=G$, in which case $H$ has a simple pole at $s=1$ \cite{MW}.

By standard arguments, this is enough to establish a Mertens-type bound for sums of $\Lambda_H(n)$.  From
\eqref{Hs} we have that
\[
 \Lambda_H(p^k) = \frac{\Lambda_F(p^k) \overline{\Lambda_G(p^k)}}{\log p} \qquad (p\nmid Q).
\]
Hence, taking $F=G$ and using \eqref{LRS}, we get the crude upper bound
\begin{align*}
 \sum_{b\le x} \frac{|\Lambda_F(b)|^2}{b} &\le \sum_{(b,Q)>1} \frac{|\Lambda_F(b)|^2}{b} 
+(e \log x) \sum_{(b,Q)=1} \frac{|\Lambda_F(b)|^2}{b^\sigma \log b} \qquad \(\text{take }\sigma=1+\frac{1}{\log x}\) \\ 
&\ll 1 + (\log x) \( - \frac{H'}{H}(\sigma) + \sum_{(b,Q)>1} |\Lambda_H(b)| b^{-\sigma} \).
\end{align*}
We have $-\frac{H'}{H}(\sigma)=\frac{1}{\sigma-1}+O(1)=\log x+O(1)$ and
\[
 \sum_{(b,Q)>1} |\Lambda_H(b)| b^{-\sigma} \le mn \sum_{p|Q} \log p \sum_{k=1}^\infty \frac{1}{p^{(\sigma-1)k}}
\ll \log x.
\]
This proves part (a) of Theorem \ref{KFG}; that is, condition (vi) holds.

To prove part (b), we observe that from the absolute convergence of the Euler product in the case $F=G$ we obtain the bound
\be\label{sumai}
\sum_{i=1}^n \sum_p \frac{|\a_i(p)|^2}{p^\sigma} < \infty \qquad (\sigma>1).
\ee
For $\Re s > 1$,
\[
 \log H(s) = \sum_{p\nmid Q} \frac{a_F(p) \overline{a_G(p)}}{p^s} + J(s) - \sum_{p|Q} \sum_{j=1}^{mn}\log(1-\gamma_j(p)p^{-s}),
\]
where
\[
 J(s) = \sum_{k=2}^\infty \frac{1}{k} \sum_{i=1}^n \sum_{j=1}^m \frac{(\a_i(p) \overline{\b_j(p)})^k}{p^{ks}}.
\]
Let $\delta=\min(\frac{1}{n^2+1},\frac{1}{m^2+1})$.  We claim that $J(s)$ is an absolutely convergent Dirichlet series (hence analytic)
for $\Re s \ge 1 - \delta/2$.  Indeed, by \eqref{LRS} and \eqref{sumai}, for such $s$,
\[
\sum_{k=2}^\infty \frac{1}{k} \sum_{i,j} \left| \frac{(\a_i(p)\b_j(p))^k}{p^{ks}} \right| \le 
\sum_{k=2}^\infty \frac{1}{k}  \sum_{i,j} \frac{p^{(1-2\del)(k-1)}}{p^{(1-\del/2)k}}|\a_i(p)|^2 < \infty.
\]
Similarly, 
\[
K_{F,G}(s)=\sum_{p\nmid Q} \frac{a_F(p) \overline{a_G(p)}\log^2 p}{p^s} + \widetilde{K}(s),
\]
where $\widetilde{K}(s)$ is analytic for $\Re s \ge 1$.  Therefore, initially for $\Re s > 1$ and for $\Re s \ge 1$ (but $s\ne 1$ if $F=G$) by analytic continuation,
\[
 K_{F,G}(s) = \(-\frac{H'}{H}\)'(s) + E(s),
\]
where $E(s)$ is analytic for $\Re s \ge 1$.  This proves part (b) of Theorem \ref{KFG}.

%%%%%%%%%%%%%%%%%%%%%%%%%

\section{Proof of Theorem \ref{main3}}

 Fix an integer $q > 1$ and a primitive Dirichlet character $\chi$ mod $q$. We also fix a
 compactly supported function $g$ in $\mathcal{C}^\infty(\mathbb{R})$ such that
 $g \ge 0$ and $\int g = 1$. Denoting
$g_m(x) = m g(mx)$,  $g_m$ is more concentrated at the origin (as $m\rightarrow\infty$, 
$g_m$ approaches a Dirac delta function),
and also
$\int g_m = 1$.
Fix real numbers $\alpha, \beta$ in $\left(\frac12 , \infty\right)$, and let 
\[
h_m(x) = g_m(x-\alpha) - g_m(x-\beta).
\]
Here the key idea is that integrating by parts $k$-times gives
\[
\int h_m^{(k)}(t) H(t) dt = \int h_m(t) H^{(k)}(t) dt.
\]
By Theorem \ref{main2} with $F = L(s,\chi)$, as $T\rightarrow\infty$,
\[
S_{m,k}(T) := \frac{2\pi}{T}\ssum{0<\g_1,\g_2 \le T \\ \rho_1,\rho_2 \in Z(L(s,\chi))} 
h_m^{(k)}(\g_1-\g_2)\rightarrow
\int h_m^{(k)}(t) H_\chi(t) dt = \int h_m(t) H_\chi^{(k)}(t) dt,
\]
where
\[
H_\chi(t) = K_\chi(1+it) + K_\chi(1-it),
\]
\[
K_\chi(s) = \sum_{n=1}^\infty \frac{|\Lambda_F(n)|^2}{n^s} \qquad (\Re s > 1),
\]
and $K_\chi(s)$ is given by analytic continuation for $\Re s = 1$.  Since $\Lambda_F(n)=\Lambda(n)\chi(n)$,
\[
K_{\chi}(s) = \sum_{(n,q)=1} \frac{\Lambda^2(n)}{n^s} = K(s) - 
\sum_{p|q}\frac{\log^2 p}{p^s - 1}\,.
\]
Since $H_{\chi}(t)\in C^{\infty}(0,\infty)$ and $h_m$ approaches the difference of
two Dirac deltas as $m\rightarrow\infty$,
\be\label{name1}
\lim_{m\rightarrow\infty}\lim_{T\rightarrow\infty} S_{m,k} (T) 
= \lim_{m\rightarrow\infty} \int h_m(t) H_{\chi}^{(k)}(t) dt 
 = H_{\chi}^{(k)}(\alpha) - H_{\chi}^{(k)}(\beta) ,
\ee
for every fixed $k$.

\begin{lem}\label{X1}
Uniformly for $\Re s = 1$, $s \ne  1 ,$
\[
\left| K_{\chi}^{(k)}(s) - K^{(k)}(s)\right| \ll k! \log q .
\]
\end{lem}

\begin{proof}

\bal
K^{(k)}(s) - K_{\chi}^{(k)}(s) 
&
= \sum_{p|q}\log^2 p\left[ (p^s-1)^{-1}\right]^{(k)}
= \sum_{p|q}(-\log p)^{k+2}\sum_{j=1}^{\infty}\frac{j^k}{p^{js}}
\\
& \ll  \sum_{p|q}(\log p)^{k+2}\int_0^\infty x^kp^{-x} dx 
\\
& = \sum_{p|q}(\log p)\cdot k! \le k! \log q.
\eal

\end{proof}

\begin{cor}\label{X2}
Uniformly for $t \ne 0 ,$
\[
H_{\chi}^{(k)}(t) - H^{(k)}(t) \ll k! \log q.
\]
\end{cor}

\begin{lem}\label{X4}
Let 
\[
W(s) = \left(\frac{\zeta'}{\zeta}\right)'(s),
\]
and let $S$ be a compact set of real numbers not containing zero.
Then, uniformly for $s = 1+it$, with $t\in S$, we have
\[
\lim_{k\rightarrow\infty}\left| \frac{K^{(k)}(s)-W^{(k)}(s)}{2^k(k+1)!}\right| = 0 .
\]
\end{lem}

\begin{proof}
First, we have
\be\label{X3}
\begin{split}
K^{(k)}(s)-W^{(k)}(s) &= (-1)^k \sum_{n=1}^{\infty} \frac{\Lambda^2(n)\log^k n - \Lambda(n)\log^{k+1}n}{n^s}\\
&= (-1)^k \sum_{l=2}^{\infty} \sum_p \frac{(l^k-l^{k+1})(\log p)^{k+2}}{p^{ls}}\\
&=:(-1)^k \sum_{l=2}^{\infty} (l^k-l^{k+1}) R_{k,l}.
\end{split}\ee

If $l \ge 3$, then
\bal
|R_{k,l}| &\le\ \int_{2^-}^\infty \frac{(\log u)^{k+1}}{u^l} d\, \theta(u)
 =  \int_{2}^\infty \theta(u) \frac{(\log u)^{k+1}}{u^l} \left( \frac{l}{u}
- \frac{k+1}{u\log u} \right)du
\\
&\ll l \int_{2}^\infty \frac{(\log u)^{k+1}}{u^l} du = \frac{l}{(l-1)^{k+2}}
\int_{(l-1)\log 2}^\infty w^{k+1} e^{-w} d w
\\
&\ll \begin{cases} \frac{l}{(l-1)^{k+2}}\cdot (k+1)! \quad ( \text {all} \; l) ,
& \\  \frac{l}{(l-1)^{k+2}} \cdot 2^{-(l-1)/2} \int_0^\infty w^{k+1} e^{-w/2} d w\ll
\frac{l}{(l-1)^{k+2}} \cdot 2^{k-l/2} (k+1)! \quad (l > 2k) .
\end{cases}
\eal
Therefore
\bal
\left|\sum_{l\ge 3} (l^k-l^{k+1}) R_{k,l}\right|
&
\ll (k+1)!\left[\sum_{l=3}^{2k}\left(\frac{l}{l-1}\right)^{k+2} + 2^k
\sum_{l=2k+1}^\infty\left(\frac{l}{l-1}\right)^{k+2}2^{-l/2} \right]
\\
& \ll (k+1)! \left\{ \left(\frac32\right)^k + k \left(\frac43\right)^k + 1\right\}
\ll \left(\frac32\right)^k (k+1)! \,.
\eal
When $l=2$, we must be more precise, using the Prime Number Theorem in the form 
$\theta(u)\sim u$ as $u\to\infty$.  Write $\theta(u)=u+E(u)$.  In what follows, $o(1)$ stands
for a function of $k$ that tends to zero as $k\to\infty$.  By partial summation we get
\[
R_{k,2} = \int_{2^-}^\infty \frac{(\log u)^{k+1}}{u^{2+2it}}du +
\int_2^ \infty\frac{(\log u)^{k+1}}{u^{3+2it}} E(u) \( 2 + 2it-\frac{k+1}{\log u}\)\, du. 
\]
Making the change of variables $w=\log u$ and using that $E(u)=o(u)$ as $u\to\infty$, we obtain
\begin{align*}
 |R_{k,2}| &\ll 1 + \left| \int_0^\infty \frac{w^{k+1}}{e^{(1+2it)w}}\, dw\right| + \int_{\log 2}^{k/10}
w^{k+1}\(1+|t|+\frac{k+1}{w}\)\, dw + o(1)\int_{k/10}^\infty \frac{w^{k+1}}{e^w}(1+|t|)\, dw \\
&= 1 + \frac{(k+1)!}{|1+2it|^{k+2}} + O((1+|t|)(k/10)^{k+2}) + o(1) (1+|t|)(k+1)! \\
&= o((k+1)!)
\end{align*}
uniformly for $t\in S$.  It follows from \eqref{X3} that $K^{(k)}(1+it)-W^{(k)}(1+it)=o(2^k(k+1)!)$
uniformly for $t\in S$. 
\end{proof}

In the same way, we can prove an analogous estimate needed for Theorem \ref{main5}.

\begin{lem}\label{X5}
Let $F,G,H,W$ be given by \eqref{FGHW}, and
and let $S$ be a compact set of real numbers not containing zero.
Then, uniformly for $s = 1+it$, with $t\in S$, we have
\[
\lim_{k\rightarrow\infty}\left| \frac{K_{F,G}^{(k)}(s)-W^{(k)}(s)}{2^k(k+1)!}\right| = 0 .
\]
\end{lem}

\begin{proof}
 We compute that $\Lambda_H(p^l)=(\a_p^l+\b_p^l)\log p$ and
\[
 \Lambda_F(p^l)\Lambda_G(p^l)=\begin{cases} \Lambda_H(p^l) \log p & p\nmid Nd \\ 0 & p|Nd. \end{cases}
\]
Hence,
\[
 (-1)^k\( W^{(k)}(s)-K_{F,G}^{(k)}(s) \) = \ssum{p,l\ge 2 \\ p\nmid Nd} \frac{(l^{k+1}-l^k)\Lambda_H(p^l)
(\log p)^{k+1}}{p^{ls}} + \sum_{p|Nd} \sum_{l=1}^\infty \frac{(l\cdot \log p)^{k+1} \Lambda_H(p^l)}{p^{ls}}.
\]
For $\Re s = 1$, the second double sum above is, in absolute value, at most
\bal
2\sum_{p|Nd} (\log p)^{k+2} \sum_{l=1}^\infty \frac{l^{k+1}}{p^l} &\ll \sum_{p|Nd} (\log p)^{k+2}
\int_0^\infty x^{k+1} p^{-x}\, dx \\
&=(k+1)! \sum_{p|Nd} 1 \ll (k+1)!.
\eal
Since $|\Lambda_H(p^l)|\le 2\log p$, the argument used in the proof of Lemma \ref{X4} implies
that the terms in the sum over $p\nmid Nd$ and $l\ge 3$ total $O((3/2)^k(k+1)!)$.  The terms
with $l=2$ and $p\nmid Nd$ total 
\[
 2^k \sum_{p\nmid Nd} \frac{(a(p)^2-2)(\log p)^{k+2}}{p^{2s}}.
\]
From the standard theory of Rankin-Selberg convolutions \cite[Theorems 5.13, 5.44]{IK}, we have $\sum_{p\le x} a(p)^2\log p \sim x$ 
(the analog of the Prime Number Theorem for $a(p)^2$) and the argument in the proof of Lemma
\ref{X4} implies that the above sum is $o(2^k(k+1)!)$, uniformly for the values of $s$ under consideration.
\end{proof}

\begin{proof}[Proof of Theorem \ref{main3}]
Denote by
\[
 H_k(t) = \frac {(-1)^{k+1}}{2^{k+1}(k+1)!} \(W^{(k)}(1+it) + W^{(k)}(1-it) \).
\]
 By \eqref{name1}, Corollary \ref{X2}, and Lemma \ref{X4},
\bal
\lim_{k\rightarrow\infty} \frac {(-1)^{k+1}}{2^{k+1}(k+1)!}
\lim_{m\rightarrow\infty}\lim_{T\rightarrow\infty} S_{m,k} (T)
&
=  \lim_{k\rightarrow\infty} \frac {(-1)^{k+1}}{2^{k+1}(k+1)!}\left( H^{(k)}(\alpha) - H^{(k)}(\beta)\right)
\\
& =  \lim_{k\rightarrow\infty}\left( H_k(\alpha) - H_k(\beta)\right) .
\eal

By the explicit formula for $\zeta'(s)/\zeta(s)$, e.g. \cite[\S 12, (8)]{Da},
\[
 W(s) = \frac{1}{(s-1)^2} - \sum_{\rho} \frac{1}{(s-\rho)^2} - \frac12 \( \frac{\Gamma'}{\Gamma}\(\frac{s}{2}+1\)\)'.
\]
Using a well-know series expansion for $\Gamma'(s)/\Gamma(s)$, we get for $k\ge 0$ that
\be\label{explicitzeta}
 W^{(k)}(s) = (-1)^k (k+1)! \( \frac{1}{(s-1)^{k+2}} - \sumstarrho \frac{1}{(s-\rho)^{k+2}} \),
\ee
where the sum on $\rho$ is over \emph{all} zeros of $\zeta$, including the trivial zeros at the points
$\rho=-2,-4,\ldots.$.

{\it Assume RH:} For any $t > \frac12$, 
\bal
\lim_{k\rightarrow\infty} H_k(t)
&
=  \lim_{k\rightarrow\infty} \frac{(-1)^{k+k+1}}{2^{k+1}} \(
\frac{1}{(it)^{k+2}} + \frac{1}{(-it)^{k+2}} - \sumstarrho
\frac1{(1+it-\rho)^{k+2}} + \frac1{(1-it-\rho)^{k+2}} \)
\\
& = \frac12 \( m_{\zeta}\(\frac12+it\) + m_\zeta\(\frac12-it\) \) = m_{\zeta}\(\frac12+it\),
\eal
where $m_{\zeta}(\rho)$ denotes the multiplicity of the zero of
the Riemann zeta function at $\rho$.  We note that the above limit calculation remains valid under the 
weaker assumption that there are no zeros off the critical line with imaginary part in $[t-\frac12,t+\frac12]$;
that is, the full strength of RH is not necessary.

This proves the implication $(b)\implies(a)$ in Theorem \ref{main3}, and it
also shows that
\[
 \lim_{k\rightarrow\infty} f_{\chi,g,\beta,k}(\alpha) = m_{\zeta}\left(\frac12 + i\alpha\right) -
 m_{\zeta}\left(\frac12 + i\beta\right).
 \]

{\it Assume RH fails:} 
Then let $\rho_0$ be a zero of the Riemann zeta function with
$\Re \rho_0 > \frac 12$. 
We distinguish two cases, according as to whether
$\zeta(s)$ has any zeros inside the open disk of radius $\frac12$ centered at $1+i\beta$.

Assume first that there are no zeros of $\zeta(s)$ inside this disk.
We then
choose $\alpha = \Im \rho_0$, and consider the zeros
of $\zeta(s)$ closest to $1+i\alpha$. Denote the minimum distance to $1+i\alpha$
by $\delta_0$, and let
$\rho_1, \dots, \rho_r$ denote the zeros of $\zeta(s)$ at distance $\delta_0$ from
$1+i\alpha$, with multiplicities $m_1,\dots,m_r$. We write the differences
$1+i\alpha- \rho_j$ in the form 
\[
1+i\alpha-\rho_j = \delta_0 e^{2\pi i\theta_j}, \; j = 1,\dots, r.
\]
Then the contribution of $\rho_1, \dots, \rho_r$ in $ f_{\chi,g,\beta,k}(\alpha)$
is of the order of magnitude of 
$\left(\frac1{2\delta_0}\right)^k\sum_{j=1}^{r} m_j  \cos 2\pi k\theta_j$.
We now let $k$ tend to infinity along a subsequence for which all the fractional parts
$\{ k\theta_j\}$ tend to zero. Then along this subsequence, 
$
\sum_{j=1}^{r} m_j  \cos 2\pi k\theta_j \rightarrow \sum_{j=1}^r m_j \,.
$
By comparison, the contribution of all the other zeros of $\zeta(s)$ in $ f_{\chi,g,\beta,k}(\alpha)$
is exponentially small. Since $\delta_0 < \frac12$,
it follows that $f_{\chi,g,\beta,k}(\alpha)$ tends to infinity
along this subsequence. 

Assume now that $\zeta(s)$ has at least one zero inside the open disk of radius $\frac12$ centered at $1+i\beta$. In this case we choose any $\alpha \in (\frac12, \infty)$ for which
all the zeros of $\zeta(s)$ are at distance larger than $\frac12$ from $1+i\alpha$.
We then repeat the reasoning from the previous case, where the role of
$\rho_1,\dots,\rho_r$ is now played by the zeros of $\zeta(s)$ closest to $1+i\beta$.
As above, it follows that $f_{\chi,g,\beta,k}(\alpha)$ tends to infinity
along a suitable subsequence. 

This completes the proof of the implication $(a)\implies(b)$ in Theorem \ref{main3}.
\end{proof}

\begin{proof}[Proof of Theorem \ref{main5}]
 This follows in the same way, using Lemma \ref{X5} instead of Lemma \ref{X4}.
Condition (vii) from the definition of the class $\SS^*$
 holds by the analysis in Section \ref{sec:Li}.  Condition (vi) holds by Theorem \ref{KFG},
and the other conditions hold by the discussion preceding the statement of Theorem \ref{main5}.
Unlike the situation in Theorem \ref{main3}, $F$, $G$ and $H$ are all entire functions
and thus $\a$ and $\b$ are not excluded from the interval $[-\frac12,\frac12]$.  This comes from
the analog of the explicit formula \eqref{explicitzeta}, where there will be no term corresponding
to the pole at $s=1$ as with $\zeta(s)$.
\end{proof}

%%%%%%%%%%%%%%%%%%%%%%%%%%%%%%%%%%%%%
%
%
\begin{section}{Notes on shifted convolution sums of Fourier coefficients of $\Gamma_0(D)$ cusp forms}
\label{sec:Li}
%
%
%%%%%%%%%%%%%%%%%%%%%%%%%%%%%%%%%%%%%
\newcommand{\Su}{\mathcal S}
\newcommand{\Saq}{\ssum{a \textup{ (mod $q$)} \\ (a,q)=1}}
\newcommand{\h}[1]{h\left(\frac qQ, \frac{#1}{Q^2}\right)}

\subsection{Zero density result}
Luo's zero density result follows from understanding a mollified second moment.  
Since $L(1/2+it, f)$ has conductor $|t|^2$, this requires understanding the shifted convolution
 problem for the Fourier coefficients $\lambda_f(n)$ of $f$ (see Lemma 2.1 of \cite{Luo}). 
 In the case of the full modular group, results were proven by Hafner \cite{Haf} and subsequently by Luo 
\cite{Luo}.  Unfortunately, there does not appear to be such a result for general congruence subgroups in the
 literature and we provide a proof below for the sake of completeness.

\subsection{Shifted convolution sum}
Fix $\mu, \nu >0$ and $l \in \mathbb{Z}$, with $l \neq 0$.  Let $\omega(x)$ be a smooth function supported on $[1/2, 5/2]$ and $N \geq 1$ (with $N \gg (\mu\nu)^\epsilon$ for convenience).  Let
\begin{equation}
\label{eqn:S}
\Su = \Su (\mu, \nu, l) := \sum_{\mu m - \nu n = l} \lambda_f(n) \lambda_f(m) \omega\pfrac{n}{N}.
\end{equation}
We have normalized the coefficients $\lambda_f$ so that $\lambda_f(n) \ll n^\epsilon$ is known by the work
 of Deligne.  We study $\Su$ using the delta method, as used by Duke, Friedlander and Iwaniec in \cite{DFI}. 
 We state the version developed by Heath-Brown \cite{HB}.  As  usual, let
$\delta(0)=1$ and $\delta(n)=0$ for $n\ne 0$.
Then there exists a smooth function $h: (0, \infty) \times \mathbb{R} \rightarrow \mathbb{R}$ such that 
$$\delta(n) = \frac{1}{Q^2} \sum_{q\geq 1} \Saq e\pfrac{an}{q} \h{n} + O_A\pfrac{1}{Q^A}.
$$Applying this to (\ref{eqn:S}) with $Q = \sqrt{\nu N}$, we see that
\begin{equation}
\Su = \frac{1}{Q^2} \sum_{q\geq 1} \Saq e\pfrac{-al}{q} \Su(a)+ O_A\pfrac{1}{N^A},
\end{equation}where
\begin{align}
\Su(a) = \sum_{m, n} \lambda_f(n) \lambda_f(m) e\pfrac{a(\mu m - \nu n)}{q} g(m, n),
\end{align}where $g(x, y)$ is a smooth function compactly supported on $[\frac{\nu N}{2\mu}+\frac l \mu, \frac{5\nu N}{2\mu}+\frac l \mu] \times [N/2, 5N/2]$, and moreover $$\frac{\partial^{j+k}}{\partial x^j \partial y^k}g(x, y) \ll_{j, k} \max\left(N^{-j}, \pfrac{\mu}{Q^2}^j\right)\max \left(\pfrac{\nu}{Q^2}^k, N^{-k}\right) \leq \frac{1}{N^{l+j}},$$ 
for all $j, l \geq 0$.

Let $\mu ' = \mu/(\mu, q)$, $q_\mu = q/(\mu, q)$, $\nu ' = \nu/(\nu, q)$, and $q_\nu = q/(q, \nu)$.  Further, let $D_{1, \mu} = (q_\mu, D)$ and let $D_{2, \mu} = D/D_{1, \mu}$ and similarly define $D_{1, \nu}$ and $D_{2, \nu}$.  By Chinese Remainder Theorem, we have that $\chi = \chi_{D_{1, \mu}}\chi_{D_{2, \mu}}$ for unique $\chi_{D_{i, \mu}}$ characters modulo $D_{i, \mu}$.  

Now we apply the Voronoi summation formula for these coefficients due to Kowalski, Michel and VanderKam (see Appendix A \cite{KMV} for the proof and exact notation) to get that
\begin{align}
\Su(a) 
&= \chi_{D_{1, \mu}}(\bar{a})\chi_{D_{1, \nu}}(\bar{a}) \chi_{D_{2, \mu}}(- q_\mu) \chi_{D_{2, \nu}}(- q_\nu) \frac{\eta_f(D_{2, \mu})\eta_f(D_{2, \nu})}{\sqrt{D_{2, \mu}D_{2, \nu}}} \notag \\
&\sum_{m, n} \lambda_{fD_{2,\mu}}(n) \lambda_{fD_{2, \nu}}(m) e\pfrac{\overline{a\mu '}m}{q_\mu} e\pfrac{- \overline{a \nu '} n}{q_\nu} G(m, n),
\end{align}where
$$G(m, n)=\frac{4\pi^2}{q_\nu q_\mu} \int_0^\infty\int_0^\infty g(x, y) J_{k-1}\pfrac{4\pi \sqrt{mx}}{q_\mu \sqrt{D_{2,\mu}}}J_{k-1}\pfrac{4\pi \sqrt{ny}}{q_\nu \sqrt{D_{2,\nu}}}dx dy.
$$The rest of the proof is similar to pg. 156 of \cite{Luo}.  To be more precise, using integration by parts and the fact that 
$$J_{k-1}(2\pi x) = \frac{1}{\pi \sqrt{x}}\Re \left(W(2\pi x) e\left(x-k/4+1/8\right)\right),$$ for an essentially flat function $W$ satisfying $x^l W^{l}(x) \ll 1$, we see that the sum over $m$ and $n$ is essentially restricted to be of length $\frac{q_{\mu}^2D_{2,\mu}\mu}{\nu N^{1-\epsilon}}$ and $\frac{q_{\nu}^2D_{2,\mu}}{N^{1-\epsilon}}$, respectively.  Moreover, this tells us that trivially
$$G(m, n) \ll \frac{1}{\sqrt{q_\nu q_\mu}(mn)^{1/4}} \left(N \frac{\nu}{\mu}N\right)^{3/4}.$$  

Also, the sum over $a$ creates a Kloosterman sum $S_\chi(l, *; q)$ for which the Weil bound may be applied.  This gives the bound 
$$\Su  \ll_f \sqrt{l} \nu^{3/4} N^{3/4+\epsilon}.
$$Note that the above bound is symmetric in $N$ and $\frac{\nu}{\mu}N$ (in the sense that it is equal to $(\frac{\nu}{\mu}N)^{3/4} \mu^{3/4}$).

One may then derive analytic continuation and bounds for the Dirichlet series
$$D_{\mu, \nu} (s, l) = \sum_{n\geq 1} \frac{a(n)a((\nu n + l)/\mu)}{(\nu n + l/2)^s}
$$ by using a smooth partition of unity.  This suffices for the purposes of proving a zero density result as in Theorem 1.1 of \cite{Luo}.

\end{section}

\bigskip

{\bf Acknowlegements.}
The authors are grateful to the referees for their helpful comments and suggestions.
The authors thank Xiannan Li for the material in Section \ref{sec:Li}, for
 helpful discussions about Rankin-Selberg $L$-functions, and for posing
the question of general linear combinations of zeros $\g-\a\g'$.
The authors are also thankful to Christopher Hughes, Jon Keating, Jeffrey Lagarias, Brad Rodgers, 
and Michael Rubinstein for informing them about previous work on this topic.

%%%%%%%%%%%%%%%%%%%%%%%%%%%%%%%%%%%%%%%%%%%%%%%%%%%%%%%%%%%%%%%%%%%%%%%%%%%%

\end{document}